\documentclass[reqno]{amsart}

\usepackage{amsfonts}
\usepackage{amssymb}
\usepackage{mathabx}
\usepackage{amscd}
\usepackage{pictexwd,dcpic}
\usepackage{graphicx}
\usepackage{xcolor}
\usepackage{tikz}

\usepackage{soul}

\usepackage{enumitem}
\usepackage{fullpage}
\usepackage{hyperref}
\hypersetup{
    colorlinks=true,
    linkcolor=blue,
    filecolor=magenta,
    urlcolor=cyan,
    citecolor=cyan,}
\usepackage[nameinlink,capitalize]{cleveref}

\definecolor{darkgreen}{RGB}{18, 80, 10}
\definecolor{gold}{rgb}{0.85,.66,0}

\title{On Rees algebras of linearly presented ideals and modules}
\author{Alessandra Costantini}
\address{Department of Mathematics, Oklahoma State University, 401 MSCS, Stillwater OK 74078}
\email{alecost@okstate.edu}
\author{Edward F. Price III}
\address{Department of Mathematics and Computer Science, Colorado College, 1112 N Nevada Ave, Colorado Springs, \indent CO 80903}
\email{eprice@coloradocollege.edu}
\author{Matthew Weaver}
\address{Department of Mathematics, University of Notre Dame, 255 Hurley Bldg, Notre Dame, IN 46556}
\email{mweaver6@nd.edu}

\date{}	
\newtheorem{thmx}{Theorem}

\newtheorem{thm}{Theorem}[section]

\newtheorem{prop}[thm]{Proposition}
\newtheorem{lemma}[thm]{Lemma}
\newtheorem{cor}[thm]{Corollary}
\numberwithin{equation}{section}

\theoremstyle{definition}
\newtheorem{rem}[thm]{Remark}
\newtheorem{set}[thm]{Setting}
\newtheorem{notat}[thm]{Notation}
\newtheorem{defn}[thm]{Definition}
\newtheorem{quest}[thm]{Question}
\newtheorem{ex}[thm]{Example}

\def\a{\mathfrak{a}}
\def\A{\mathcal{A}}
\def\J{\mathcal{J}}
\def\K{\mathcal{K}}
\def\L{\mathcal{L}}
\def\m{\mathfrak{m}}
\def\p{\mathfrak{p}}
\def\q{\mathfrak{q}}
\def\R{\mathcal{R}}
\def\S{\mathcal{S}}
\def\F{\mathcal{F}}

\def\ann{\mathop{\rm ann}}

\def\dim{\mathop{\rm dim}}
\def\fitt{\mathop{\rm Fitt}}
\def\grade{\mathop{\rm grade}}
\def\depth{\mathop{\rm depth}}
\def\hgt{\mathop{\rm ht}}

\def\spec{\mathop{\rm Spec}}

\def\rank{\mathop{\rm rank}}

\begin{document}
\maketitle

\begin{abstract}
   Let $I$ be a perfect ideal of height two in $R=k[x_1, \ldots, x_d]$ and let $\varphi$ denote its Hilbert-Burch matrix. When $\varphi$ has linear entries, the algebraic structure of the Rees algebra $\R(I)$ is well-understood under the additional assumption that the minimal number of generators of $I$ is bounded locally up to codimension $d-1$. In the first part of this article, we determine the defining ideal of $\R(I)$ under the weaker assumption that such condition holds only up to codimension $d-2$, generalizing previous work of P.~H.~L.~Nguyen. In the second part, we use generic Bourbaki ideals to extend our findings to Rees algebras of linearly presented modules of projective dimension one.
\end{abstract}

\section{Introduction}
 Given an ideal $I$ in a Noetherian ring $R$, the Rees algebra of $I$ is the graded algebra $\R(I) \coloneq R[It] =R\oplus It \oplus I^2t^2 \oplus \cdots$ viewed as a subring of $R[t]$, where $t$ is an indeterminate. As $\R(I)$ encodes information on the asymptotic growth of the powers of $I$, it has proven to be an invaluable asset to commutative algebraists within the study of integral closure, reductions, and multiplicities. From a geometric perspective, $\mathrm{Proj}(\R(I))$ is the blowup of an affine scheme along $V(I)$, hence Rees algebras represent an essential tool to understand singularities. 
 In fact, a celebrated result of Hironaka \cite{Hironaka} uses repeated blowups to construct resolutions of singularities in characteristic zero. As successive blowups of a singular scheme along disjoint subschemes $V(I_1), \ldots, V(I_r)$ correspond to the Rees algebra $\R(I_1 \oplus \cdots \oplus I_r)$ of the module $I_1 \oplus \cdots \oplus I_r$, one is motivated to study Rees algebras of modules as well. Moreover, Rees algebras of both ideals and modules arise as homogeneous coordinate rings of graphs of rational maps, thus understanding their algebraic structure gives insight into the study of birationality of algebraic varieties.

 As the blowup construction describes $\mathrm{Proj}(\R(I))$ via parametric equations, a fundamental problem is to find the implicit equations of blowups. This is a special case of the so-called \textit{implicitization problem}, namely the replacement of parametric equations with a system of implicit equations representing the same geometric object. In dimensions $d=2$ and $d=3$, this problem is referred to as the \textit{moving curve} and \textit{moving surface} problem respectively \cite{SGD97}, with a wealth of applications to geometric model theory and computer-aided design. Algebraically, this translates to finding a presentation of the Rees algebra $\R(I)$ as a quotient of a polynomial ring over $R$. The ideal $\J$ defining this quotient ring is called the \textit{defining ideal} of the Rees algebra. The generators of this ideal, the \textit{defining equations} of $\R(I)$, are often elusive and difficult to determine. For instance, the general problem of determining the defining ideal of Rees algebras remains open even for three-generated ideals in $k[x,y]$ \cite{Buse}.

 Whereas $\R(I)$ encodes data regarding the powers of $I$, its defining ideal $\J$ encodes data regarding the \textit{syzygies} of every power of $I$. Hence one often restricts to classes of ideals with a rich algebraic structure, specifically those with a well-understood free resolution. In particular, Rees algebras of grade-two perfect ideals and grade-three perfect Gorenstein ideals have been studied extensively under various assumptions using the Hilbert-Burch theorem \cite[20.15]{Eisenbud} and the Buchsbaum-Eisenbud theorem \cite{BE77} respectively (see e.g. \cite{{BM16},{DRS18},{GST18},{KPU17},{MU96},{Nguyen14},{Nguyen17},{RS14},{SUV93},{UV93},{Weaver1},{Weaver2}}).

 In the first part of this paper, we consider a linearly presented perfect ideal $I$ of height two in $R=k[x_1, \ldots, x_d]$, where $k$ is a field. Most of the existing literature in this setting analyzes the structure of the Rees algebra $\R(I)$ under the additional assumption that $I$ satisfies the so-called \textit{$G_d$ condition}, introduced by Artin and Nagata in \cite{AN72}. For a positive integer $s$, $I$ is said to satisfy $G_s$ if, for every $\p \in V(I)$ with $\dim R_{\p} \leq s-1$, $I_{\p}$ is minimally generated by at most $\dim R_{\p}$ elements. If $I$ satisfies $G_s$ for all $s$, then $I$ is said to satisfy $G_\infty$. Equivalently, $I$ satisfies $G_s$ if the $i^{\text{th}}$ Fitting ideal of $I$, $\fitt_i(I)$, satisfies $\hgt \fitt_i(I) \geq i+1$, for all $1 \leq i \leq s-1$. This condition thus imposes a constraint on a presentation matrix of $I$, which when $s=d$, crucially simplifies the problem of determining the defining ideal of $\R(I)$ \cite{{BM16},{MU96},{SUV93},{UV93}}.  
 
Whereas the $G_d$ condition is certainly an asset in the search for the defining equations, there are many ideals of great interest which do not satisfy this condition. For example, many determinantal ideals of generic height in Noetherian rings of sufficiently large dimension $d$ do not satisfy $G_{d}$; see \cite[3.3--3.5]{CP22}. Moreover, there are several examples of height-two perfect ideals in $R=k[x_1, \ldots, x_d]$ which do not satisfy $G_d$, even under the additional assumption of linear presentation. For instance, given $t$ linear forms $f_1, \ldots, f_t$ with $t \geq d \geq 3$, the ideal generated by all $(t-1)$-fold products of the $f_i$'s is a linearly presented perfect ideal of height two so long as any three of the $f_i$'s form a regular sequence, but only satisfies $G_d$ if any $d$ of the $f_i$'s do; see \cite[4.5]{CDG22}. Moreover, a linearly presented ideal $I \subset k[x,y,z]$ which is an intersection of powers of height-two prime ideals is always height-two perfect, but satisfies $G_3$ if and only if it is radical; \cite[4.2]{DRS18}. As such, when searching for the defining equations of the Rees algebra, there is significant motivation and necessity to develop techniques for situations where the $G_{d}$ condition is not satisfied, even within the well-studied class of height-two perfect ideals in $R=k[x_1, \ldots, x_d]$.

Inspired by work of P.~H.~L.~Nguyen \cite{{Nguyen14},{Nguyen17}}, we study the Rees algebra of linearly presented height-two perfect ideals which satisfy $G_{d-1}$ but not $G_{d}$. Our main result for Rees algebras of ideals is the following.

\begin{thmx}[{\rm \Cref{defideal}}] \label{introidealtheorem}
   Let $R=k[x_1,\ldots,x_d]$ be a standard-graded polynomial ring of dimension $d\geq 3$ over a field $k$, and let $I$ be a perfect $R$-ideal of height $2$ generated by $n \geq d+1$ elements, which satisfies $G_{d-1}$, but not $G_d$. Assume that $I$ has a presentation matrix $\varphi$ consisting of linear entries in $R$ with $I_1(\varphi) = (x_1,\ldots,x_d)$ and that (after possibly a change of coordinates) the matrix $\varphi$ has rank 1 modulo $(x_1, \ldots, x_{d-1})$. The Rees algebra $\R(I) \cong R[T_1, \ldots, T_n]/\J$ is a Cohen-Macaulay domain, with $$\J = (\ell_1, \ldots, \ell_{n-1}) : (x_1,\ldots,x_{d-1}) = (\ell_1, \ldots, \ell_{n-1}) + I_{d}(M),$$ where $\ell_1, \ldots, \ell_{n-1}$ are linear forms in $R[T_1, \ldots, T_n]$ such that $[\ell_1 \ldots \ell_{n-1}]= [T_1 \ldots T_n] \cdot \varphi$, and $M$ is a $d\times (n-1) $ matrix with entries in $k[T_1, \ldots, T_n]$ such that $[\ell_1\ldots \ell_{n-1}] = [x_1\ldots x_{d-1} \,\, x_dT_n] \cdot M$.
\end{thmx}

 This theorem recovers \cite[5.1~and~5.2]{Nguyen14} in the case that $n=d+1$, as well as \cite[4.6]{Nguyen17} when $d=3$ (see also \cite[3.6]{DRS18} for a more detailed statement when $d=3$). While our methods build upon Nguyen's techniques, a key novel idea within the proof of \Cref{introidealtheorem} is to interpret $\J$ as a prime saturation (see \Cref{Jasat}), using methods developed in \cite{BM16} in a different context. This allow us to remove the assumption that $k$ is algebraically closed in \cite{Nguyen14,Nguyen17}, as well as any constraints on the number of variables; see also \Cref{d+epres}. 
 The matrix $M$ appearing in the statement above is constructed from the so-called \textit{Jacobian dual matrix} $B(\varphi)$ of $\varphi$, introduced by Vasconcelos in \cite{Vasconcelos91}; see \cref{prelimsection} for the technical definition. Moreover, the ideal $I_{d}(M)$ coincides with  $I_{d-1}(B')$, where $B'$ is a submatrix of $B(\varphi)$ obtained by removing one row and one column. We emphasize the similarity of our result with the case of linearly presented height-two perfect ideals satisfying $G_d$, for which the non-linear generators of $\J$ are precisely the elements of $I_{d}(B(\varphi))$ \cite[1.3]{MU96}.

As ideals of positive grade can be identified with torsion-free modules of rank one, it is natural to ask whether \Cref{introidealtheorem} can be extended to the case of Rees algebras of modules of arbitrary rank. We address this question in the second part of this article. Since the ideals in \Cref{introidealtheorem} are perfect of grade two, we consider modules with projective dimension one. By imposing similar restrictions on the presentation matrix of such a module $E$, we obtain an analogous description for its Rees algebra $\R(E)$. Our main result for Rees algebras of modules is the following.

  \begin{thmx}[{\rm \Cref{defidealmodule}}] \label{intromoduletheorem}
   Let $R=k[x_1,\ldots,x_d]$ be a standard-graded polynomial ring of dimension $d\geq 3$ over a field $k$, and let $E$ be a module of projective dimension one and rank $e>0$. Assume that $E$ is generated by $n \geq d+e$ elements and satisfies $G_{d-1}$, but not $G_d$. Furthermore, suppose that $E$ has a presentation matrix $\varphi$ consisting of linear entries in $R$ with $I_1(\varphi) = (x_1,\ldots,x_d)$ and that (after possibly a change of coordinates) the matrix $\varphi$ has rank 1 modulo $(x_1, \ldots, x_{d-1})$. The Rees algebra $\R(E) \cong R[T_1, \ldots, T_n]/\J$ is a Cohen-Macaulay domain, with $$\J = (\ell_1, \ldots, \ell_{n-e}) : (x_1,\ldots,x_{d-1}) = (\ell_1, \ldots, \ell_{n-e}) + I_{d}(M),$$ where $\ell_1, \ldots, \ell_{n-e}$ are linear forms in $R[T_1, \ldots, T_n]$ such that $[\ell_1 \ldots \ell_{n-e}]= [T_1 \ldots T_n] \cdot \varphi$, and $M$ is a $d\times (n-e) $ matrix with entries in $k[T_1, \ldots, T_n]$ such that $[\ell_1\ldots \ell_{n-e}] = [x_1\ldots x_{d-1} \,\, x_dT_n] \cdot M$.
\end{thmx}

We note that the $G_s$ condition for modules is defined similarly as for ideals; see \Cref{defGs}. The crucial ingredient within the proof of \Cref{intromoduletheorem} is the existence of a \textit{generic Bourbaki ideal} $I$ of $E$, which allows us to reduce the problem to the case of Rees algebras of ideals. Generic Bourbaki ideals (which we recall in detail in \Cref{prelimsection}) were introduced by Simis, Ulrich, and Vasconcelos in \cite{SUV03} to study the Cohen-Macaulay property of Rees algebras of modules. Under the assumptions of \Cref{intromoduletheorem}, we prove that a generic Bourbaki ideal $I$ of $E$ satisfies the assumptions of \Cref{introidealtheorem} (see \Cref{choosegbi}). Thanks to the Cohen-Macaulay property of $\R(I)$, we then show that the defining ideal of $\R(E)$ has the same shape as the defining ideal of $\R(I)$. This technique was first used in \cite[4.11]{SUV03} for linearly presented modules of projective dimension one satisfying $G_d$, in which case $\R(I)$ is Cohen-Macaulay as well. We refer the curious reader to \cite[5.6]{Costantini} and \cite[6.3]{Weaver1} for cases when this Cohen-Macaulay assumption is not satisfied, yet similarly one can determine the defining ideal of the Rees algebra of a module by reducing to the case of ideals.

 This paper is organized as follows. In \Cref{prelimsection}, we recall the necessary preliminary material on Rees algebras of ideals and modules, as well as the fundamental properties and construction of generic Bourbaki ideals. We also include a short overview of residual intersections of ideals, which are briefly required in the proof of \Cref{introidealtheorem} (see \Cref{cmhgt}). \Cref{idealsection} is dedicated to linearly presented height-two perfect ideals, with main result \Cref{defideal}. The essential ingredient in its proof is \Cref{minFitt}, where we show that the $(d-1)^{\text{st}}$ Fitting ideal of $I$ has a unique minimal prime, allowing the defining ideal of $\R(I)$ to be realized as a prime saturation in \Cref{Jasat}. We end the section by presenting some examples showing that the conclusion of \Cref{introidealtheorem} should not expected to hold if one modifies the assumptions. Lastly, in \Cref{modulesection} we study the Rees algebra of a linearly presented module $E$ of projective dimension one, with main result \Cref{defidealmodule}. The key ingredient to its proof is the existence of a generic Bourbaki ideal $I$ satisfying the main result of \Cref{idealsection}. From there,  we relate the Rees algebras $\R(E)$ and $\R(I)$, the defining ideal of the latter being known from \Cref{defideal}.

\section{Preliminaries} \label{prelimsection}
In this section, we collect some background notions which will be used throughout the paper. We primarily consider ideals and modules over a local ring $R$, but we note that each of the conventions presented in this setting apply to homogeneous ideals and modules over a standard-graded ring $S$ with $S_0$ a local ring, as well.

\subsection{Fitting Ideals}

We briefly recall the notion of the Fitting ideals of a finitely generated module $E$ over a Noetherian ring $R$. Suppose that $E$ has a free presentation 
$$R^m \overset{\varphi}{\longrightarrow} R^n \longrightarrow {E} \longrightarrow 0$$
and let $I_t(\varphi)$ denote the ideal of $t\times t$ minors of $\varphi$.
\begin{defn}
The $i^{\text{th}}$ \textit{Fitting ideal} of $E$ is the ideal $\fitt_i(E) = I_{n-i}(\varphi)$.
\end{defn}

These determinantal ideals are particularly useful in that they define certain invariants of a module $E$, in terms of a presentation.

\begin{prop}[{\cite[20.4--20.6]{Eisenbud}}]\label{fitt}
 With $\varphi$ a presentation of $E$ as above, we have the following.

 \begin{itemize}
     \item[{\rm(a)}] $\fitt_i(E)$ does not depend on the choice of presentation $\varphi$ and depends only on the module $E$ and the index $i$.
     \item[{\rm(b)}] If $R$ is local, then $\fitt_i(E) = R$ if and only if $\mu(E) \leq i$.
     \item[{\rm(c)}] $V(\fitt_i(E)) = \{\p \in \spec(R)\,|\, \mu(E_\p)\geq i+1\}$.
 \end{itemize}
\end{prop}

Recall that $E$ is said to have a \textit{rank} if $E \otimes_R \mathrm{Quot}(R) \cong \mathrm{Quot}(R)^e$ for some integer $e$, where $\mathrm{Quot}(R)$ denotes the total ring of quotients of $R$. Here we say $\rank E=e$.

\begin{defn}\label{defGs}
A module $E$ with $\rank E =e$ is said to satisfy $G_s$ if $\mu(E_{\p}) \leq \dim R_{\p} +e -1$ for every $\p \in \spec(R)$ with $1\leq \dim R_{\p} \leq s-1$. If $E$ satisfies $G_s$ for all $s$, $E$ is said to satisfy $G_\infty$.
\end{defn}
In particular, when $e=1$, this recovers the definition given for ideals in the introduction. Additionally notice that from \Cref{fitt}, we equivalently have that $E$ satisfies $G_s$ if $\hgt \fitt_i(E) \geq i-e+2$ for all $e\leq i\leq s+e-2$.

We next present a lemma which will be used frequently throughout this paper in the case $s=d-1$. We state it more generally for every integer $s$, in light of potential future applications; see \Cref{Gssubsection}.

\begin{lemma}\label{hgtlfittlemma}
Let $R$ be a Noetherian ring and $E$ a finitely generated $R$-module with rank $e$ satisfying $G_s$ but not $G_{s+1}$. Then $\hgt \fitt_{s+e-2}(E)= \hgt \fitt_{s+e-1}(E) =s$.
\end{lemma}

\begin{proof}
As $E$ satisfies $G_s$, it follows that $\hgt \fitt_{s+e-2}(E) \geq s$. Moreover, since $E$ does not satisfy $G_{s+1}$, we have that $\hgt \fitt_{s+e-1}(E) \leq s$. Now since $\fitt_{s+e-2}(E) \subseteq \fitt_{s+e-1}(E)$, the claim follows.
\end{proof}

\subsection{Residual intersections}

Originating within the study of intersections of algebraic varieties, residual intersections of an ideal $I$ have become a useful tool to understand the powers of $I$ \cite{U94}. To keep this article self-contained, we recall only the basic definitions and a few results from the literature which we will use in the proof of \Cref{cmhgt}. 

\begin{defn}[{\cite[1.1]{HU88}}] \label{defResInt}
Let $R$ be a Cohen-Macaulay local ring, $I$ an $R$-ideal, and $s$ an integer with $s\geq \hgt I$. A proper ideal $J=\a:I$, where $\a = (a_1,\ldots,a_s) \subseteq I$, is said to be an \textit{$s$-residual intersection} of $I$ if $\hgt J\geq s$.
\end{defn}

Residual intersections of \textit{complete intersection ideals} are a very interesting case and have been well-studied \cite{BKM90}. In particular, through their study in the generic case, any residual intersection of a complete intersection ideal can be realized as the specialization of a generic residual intersection \cite[1.5]{HU90}. We refer the reader to \cite{{HU88},{HU90}} for a thorough treatment on generic residual intersections, which would be beyond the scope of this paper. For our purposes, we will only need the following result.

\begin{thm}\label{thmResInt}
Let $R$ be a Cohen-Macaulay local ring and $I=(x_1,\ldots,x_g)$ an $R$-ideal with $x_1,\ldots,x_g$ an $R$-regular sequence. Let $\a=(a_1,\ldots,a_s) \subseteq I$ such that $J=\a:I$ is an $s$-residual intersection of $I$.
\begin{itemize}
    \item[{\rm(a)}] {\rm(}\cite[1.4 and 1.5]{HU90}{\rm)} $R/J$ is a Cohen-Macaulay ring.    
    \item[{\rm(b)}] {\rm(}\cite[1.5 and 1.8]{HU90}{\rm)} If $B$ is a $g\times s$ matrix with $[a_1 \ldots a_s] = [x_1\ldots x_g]\cdot B$, then $J =\a + I_g(B)$.
\end{itemize}
\end{thm}

\subsection{Blowup algebras}

As both Rees algebras of ideals and modules are considered in this article, we proceed as generally as possible in this subsection and choose our notation so as to treat both cases simultaneously. Assume that $R$ is a Noetherian ring and let $E = Ra_1+\cdots +Ra_n$ be a finitely generated $R$-module with rank $e>0$. As torsion-free modules of rank one are isomorphic to ideals of positive grade, the definitions given below pertain to ideals as well, in the case that $e=1$. 

There is a natural homogeneous epimorphism of graded $R$-algebras $$\eta \colon R[T_1,\ldots,T_n] \twoheadrightarrow \S(E)$$ given by $T_i\mapsto a_i\in [\S(E)]_1$, where $\S(E)$ is the \textit{symmetric algebra} of $E$. The kernel $\L$ of this map is easily described from a presentation of $E$. Indeed, if $R^m\overset{\varphi}{\rightarrow}R^n \rightarrow E\rightarrow 0$ is any presentation of $E$, then $\L$ is generated by linear forms $\ell_1,\ldots,\ell_m$ such that $[T_1\ldots T_n]\cdot \varphi = [\ell_1 \ldots \ell_m]$. Naturally, there is an induced isomorphism $\S(E) \cong R[T_1,\ldots,T_n]/\L$.

\begin{defn}
The \textit{Rees algebra} of $E$ is defined as $\S(E)/\tau(\S(E))$, where $\tau(\S(E))$ is the $R$-torsion submodule of $\S(E)$.    
\end{defn}

We note that if $E$ is isomorphic to an ideal, say $E\cong I$, then the Rees algebra $\mathcal{R}(E)$ defined above is isomorphic to the subalgebra $\R(I) = R[It] \subseteq R[t]$, as defined in the introduction.

Composing the map $\eta$ above with the natural map factoring the $R$-torsion of $\S(E)$, one obtains a homogeneous epimorphism $$\pi \colon R[T_1,\ldots,T_n] \twoheadrightarrow \R(E)$$ given by $T_i\mapsto a_i\in [\R(E)]_1$. The kernel $\J = \mathrm{ker}(\pi)$ is the \textit{defining ideal} of $\R(E)$ and there is an induced isomorphism $\R(E) \cong R[T_1,\ldots,T_n]/\J$. By construction, it is clear that $\L\subseteq \J$.

\begin{defn}
If $\L = \J$, or equivalently $\S(E) \cong \R(E)$, then $E$ is said to be of \textit{linear type}. 
\end{defn}

The terminology is fitting as $\L$ is generated in degree one and moreover, $\L = [\J]_1$. If $E$ is not of linear type, a typical source of higher-degree generators of $\J$ is the \textit{Jacobian dual matrix}.

\begin{defn}
Let $R^m\overset{\varphi}{\rightarrow}R^n \rightarrow E\rightarrow 0$ be a presentation of $E$ and $\ell_1,\ldots,\ell_m$ the generators of $\L$ as before. There exists an $r\times m$ matrix $B(\varphi)$ consisting of linear entries in $R[T_1,\ldots,T_n]$ with
$$[T_1 \ldots T_n] \cdot \varphi=[\ell_1 \ldots \ell_m] = [x_1\ldots x_r]\cdot B(\varphi) $$
where $(x_1,\ldots,x_r)$ is an ideal containing $I_1(\varphi)$. We say that $B(\varphi)$ is a \textit{Jacobian dual} matrix of $\varphi$ with respect to the sequence $x_1,\ldots,x_r$.    
\end{defn}

 It is important to note that $B(\varphi)$ is not unique in general. However, if $R=k[x_1,\ldots,x_d]$ and $I_1(\varphi) \subseteq (x_1,\ldots,x_d)$, there is a unique Jacobian dual matrix $B(\varphi)$ with respect to $x_1,\ldots,x_d$ if and only if the entries of $\varphi$ are linear; see \cite[p.~47]{SUV93}. In this case, the entries of $B(\varphi)$ belong to the subring $k[T_1,\ldots,T_n]$, a fact that will be crucial in several arguments presented in this paper.

We next recall the definition of another graded algebra, which is closely related to $\R(E)$.

\begin{defn}
Further assume that $R$ is a local ring with maximal ideal $\m$ and residue field $k$. The \textit{special fiber ring} of $E$ is 
$$\F(E) := \R(E)\otimes_R k \cong \R(E)/\m \R(E).$$ 
The Krull dimension of the special fiber ring is called the \textit{analytic spread} of $E$ and is denoted by $\ell(E)$.    
\end{defn}

\subsection{Generic Bourbaki ideals} 
Introduced by Simis, Ulrich, and Vasconcelos in \cite{SUV03}, generic Bourbaki ideals have proven to be a powerful tool within the study of Rees algebras of modules. Indeed, their implementation often allows one to reduce problems regarding Rees algebras of modules to the case of ideals, where information is much more readily available. We recall their construction below, as well as some of their useful properties. 

 \begin{notat}\label[notat]{NotationBourbaki} 
   Let $R$ be a Noetherian ring, $E=Ra_1 + \dots + Ra_n$ a finitely generated $R$-module with $\rank_{\,}E=e>0$. Let $\displaystyle Z= \{ Z_{ij} \, | \, 1 \leq i \leq n$, $1\leq j \leq e-1 \}$ be a set of indeterminates, and denote $$ R'\coloneq R[Z], \quad E'\coloneq E \otimes_R R', \quad y_j \coloneq \sum_{i=1}^n Z_{ij} a_i \in E',\quad \mathrm{and} \quad F'\coloneq \sum_{j=1}^{e-1} R' y_j. $$ If $R$ is local with maximal ideal $\m$, let $\,\displaystyle R'' \coloneq R(Z)= R[Z]_{\m R[Z]}$ and similarly denote $\,\displaystyle E''\coloneq E \otimes_R R''\,$ and $\, \displaystyle F''\coloneq F' \otimes_{R'} R''$.
 \end{notat}

 \begin{thm}[{\rm \cite[3.1--3.4]{SUV03}}]  \label{existBourbaki} 
   Let $R$ be a Noetherian local ring, $E$ a finitely generated $R$-module with $\rank_{\,}E=e>0$. Assume that $E$ is torsion-free and that $E_{\p}$ is free for all $\p \in \spec(R)$ with $\depth R_\p \leq 1$ (e.g., if $E$ satisfies $G_2$). 
   \begin{itemize}
      \item[{\rm(a)}] For $R''$, $E''$, and $F''$ as in \cref{NotationBourbaki}, $\,F''$ is a free $R''$-module of rank $e-1$ and $E''/F''$ is isomorphic to an $R''$-ideal $I$ with $\grade I >0$. Moreover, $I$ can be chosen to have grade at least $\,2\,$ if and only if $E$ is orientable.
      \item[{\rm(b)}] The ideal $I$ satisfies $G_s$ if and only if $E$ satisfies $G_s$.     
      \item[{\rm(c)}] If $\grade I \geq 3$, then $\,E \cong R^{e-1} \oplus L\,$ for some $R$-ideal $L$. In this case, $LR'' \cong I$.
    \end{itemize}
 Such an ideal $I$ is called a \textit{generic Bourbaki ideal} of $E$. Moreover, if $K$ is another ideal constructed this way using variables $Y$, then the images of $I$ and $K$ in $S=R(Z,Y)$ coincide up to multiplication by a unit in $\mathrm{Quot}(S)$, and are equal whenever $I$ and $K$ have grade at least two.
 \end{thm}   

 Recall that a module $E$ of rank $e >0$ is said to be \textit{orientable} if $\, (\bigwedge^{e} E)^{\ast \ast} \cong R, \,$ where $\bigwedge^{e} E$ is the $e^{\text{th}}$ exterior power of $E$ and $(-)^{\ast}$ denotes the functor $\mathrm{Hom}_R (-, R)$. 
 This condition is not restrictive in our setting, as it is automatically satisfied for any module of finite projective dimension. Moreover, every finitely generated module over a unique factorization domain is orientable.

 The effectiveness of generic Bourbaki ideals in the study of Rees algebras of modules is illustrated by the following fundamental result, which we will refer to often in \Cref{modulesection}.

 \begin{thm}[{\rm \cite[3.5 and 3.8]{SUV03}}]\label{MainBourbaki} 
   In the setting of \cref{NotationBourbaki} and with the assumptions of \cref{existBourbaki}, the following statements are true.
     \begin{itemize}
       \item[{\rm(a)}] $\R(E)$ is Cohen-Macaulay if and only if $\,\R(I)$ is Cohen-Macaulay.
       \item[{\rm(b)}] If $\,\grade \, \R(E)_+ \geq e$, then $\,\R(I) \cong \R(E'')/(F'')$. 
       \item[{\rm(c)}] $E$ is of linear type with $\,\grade \, \R(E)_+ \geq e\,$ if and only if $I$ is of linear type.
    \end{itemize}
   Moreover, if any of the conditions above are satisfied, then $\,\R(I) \cong \R(E'')/(F'')$ and $y_1, \ldots, y_{e-1}$ form a regular sequence on $\R(E'')$. 
\end{thm}

The latter statement says that $\R(E'')$ is a \textit{deformation} of $\R(I)$. When this happens (e.g., when assumption (b) in the proceeding theorem is satisfied), the Cohen-Macaulay property of the special fiber ring also transfers back and forth between the module $E$ and a generic Bourbaki ideal $I$ of $E$. 

\begin{thm}[{\rm \cite[4.7 and 4.8]{Costantini}}]\label{CMfiber}
    In the setting of \cref{NotationBourbaki} and with the assumptions of \cref{existBourbaki}, the following statements are true. 
    \begin{itemize}
      \item[{\rm(a)}] If $\F(E)$ is Cohen-Macaulay, then $\F(I)$ is Cohen-Macaulay.
      \item[{\rm(b)}] Assume that either $\, \R(I)$ is $S_2$, or $\,\depth_{\,} \R(I_{\q}) \geq 2\,$ for all $\q \in \spec(R'')$ so that $I_{\q}$ is not of linear type. If $\F(I)$ is Cohen-Macaulay, then $\F(E)$ is Cohen-Macaulay, $\F(I) \cong \F(E'')/(F'')$ and $y_1, \ldots, y_{e-1}$ form a regular sequence on $\F(E'')$.
    \end{itemize}
 \end{thm}

 \begin{rem}[{\cite[p. 617]{SUV03}, \cite[5.1]{Costantini}}]\label{gbipresentation}
Let $(R,\m,k)$ be a Noetherian local ring and $E$ an $R$-module as in \Cref{existBourbaki}. Furthermore, let $R^m\overset{\varphi}{\rightarrow}R^n \rightarrow E\rightarrow 0$ be a minimal presentation of $E$.
\begin{itemize}
    \item[(a)] With $Z$ and $y_j$ as in \Cref{NotationBourbaki}, we may assume that $E''$ is generated by the images of $\,y_1,\ldots,y_{e-1}$, $a_e,\ldots,a_n$ in $E''$. Moreover, after multiplying $\varphi$ by an invertible matrix with entries in $k(Z)$, we may assume that $\varphi$ presents $E''$ with respect to this generating set. In particular, we have that
    $$\varphi= \begin{bmatrix} \hspace{2mm}A\hspace{2mm} \\
    \psi\end{bmatrix}$$
where $A$ and $\psi$ are submatrices of sizes $(e-1)\times m$ and $(n-e+1) \times m$ respectively. 
By construction, it follows that, if $E$ is torsion-free, then $\mathrm{Coker}(\psi)$ is isomorphic to a generic Bourbaki ideal $I$ of $E$. Moreover, $\psi$ is a minimal presentation matrix of $I$ since $\mu(I) = n-e+1$.
    \item[(b)] In particular, if $R=k[x_1, \ldots, x_d]$ and $E$ has a linear presentation, then $I$ does as well.
\end{itemize}
\end{rem}

\section{Height two perfect ideals with a linear presentation}\label{idealsection}

In this section, we aim to determine the defining ideal $\J$ of the Rees algebra of a grade two perfect ideal $I \subseteq k[x_1, \ldots, x_d]$, under the assumption that $I$ does not satisfy the $G_d$ condition, but rather the weaker condition $G_{d-1}$. More precisely, our working setting is the following. 

\begin{set}\label{idealsetting}
    Let $R=k[x_1,\ldots,x_d]$ be a standard-graded polynomial ring over a field $k$, with $\m = (x_1,\ldots,x_d)$ and $d\geq 3$. Let $I$ be a perfect $R$-ideal of height $2$ generated by $\mu(I) = n\geq d+1$ many elements, satisfying the following assumptions:

    \begin{itemize}
        \item[(i)] The ideal $I$ has a presentation matrix $\varphi$ consisting of linear entries with $I_1(\varphi) = \m$.
        \item[(ii)] After possibly a change of coordinates, the matrix $\varphi$ has rank 1 modulo an ideal generated by $d-1$ variables.
        \item[(iii)] The ideal $I$ satisfies $G_{d-1}$, but not $G_d$.
    \end{itemize}
\end{set}

When $d=3$ our setting coincides with the setting of \cite{Nguyen17}, while when $n=d+1$ it recovers the setting of \cite{Nguyen14}. In both settings of \cite{Nguyen14,Nguyen17} it is assumed that $k$ is an algebraically closed field, an assumption we may omit in our situation. Moreover, from \Cref{d+epres} and the discussion prior to its statement, this assumption may also be omitted in \cite{Nguyen14,Nguyen17}. In particular, assumption (ii) can be removed when $\mu(I) = d+1$, as it is implied by the remaining conditions; see \cite[3.3]{Nguyen14} and \Cref{d+epres}.

\begin{rem}\label{datleast3}
It is assumed that $d\geq 3$ in \Cref{idealsetting}, as otherwise assumption (iii) would never be satisfied by an ideal of height two. Indeed, any ideal $I$ with $\hgt I= g$ satisfies $G_{g}$ automatically. Moreover, we assume that $\mu(I) \geq d+1$ in \Cref{idealsetting}, else the remaining assumptions are contradicted and the defining ideal of $\R(I)$ is already known.
Indeed, if $\mu(I) \leq d-1$ , then since $I$ satisfies $G_{d-1}$, it follows that $I$ satisfies $G_d$ as well, contradicting assumption (iii). Moreover, it then follows that $I$ satisfies $G_\infty$ and is actually of linear type by \cite[9.1]{HSV81}.  If $\mu(I) =d$, then assumption (iii) implies that $\hgt \fitt_{d-1}(I) = d-1$ by \Cref{hgtlfittlemma}. However, this is precisely the ideal of entries of $\varphi$, $\fitt_{d-1}(I)= I_1(\varphi)$, thus contradicting assumption (i). Moreover, it then follows that $I$ belongs to a polynomial subring $R'$ of $R$ with $\dim R'=d-1$. The defining ideal of $\R(I)$ is then known due to \cite[1.3]{MU96}.
\end{rem}

\begin{rem}\label{idealpres} 
Following assumption (ii) in \Cref{idealsetting}, we may assume that $\varphi$ has rank one modulo the ideal $(x_1,\ldots,x_{d-1})$. We proceed supposing that the appropriate change of coordinates has been made and this is the ideal in assumption (ii). Let $\overline{\varphi}$ denote the image of $\varphi$ modulo $(x_1,\ldots,x_{d-1})$, with entries in $k[x_d]$. After potential row and column operations, we may assume $\overline{\varphi}$ has a single nonzero entry, which can be assumed to be $x_d$ by linearity. Thus the presentation matrix $\varphi$ has the form 
\begin{equation}\label{phiform}
 \varphi =  \begin{pmatrix}
   & & & *\\
   & \varphi'& & \vdots\\
   & & & *\\
    *&\cdots& *& x_d
\end{pmatrix}  
\end{equation}
where the entries of $\varphi'$ and the ``$*$" entries are linear and contained in the subring $k[x_1,\ldots,x_{d-1}] \subset R$. 
\end{rem}

With assumption (ii) in \Cref{idealsetting}, not only does $\varphi$ have a particular form, but its Jacobian dual $B(\varphi)$ does as well. Indeed, as a consequence of \Cref{idealpres}, the Jacobian dual of $\varphi$, with respect to $x_1,\ldots,x_d$, is 
\begin{equation}\label{JD}
 B(\varphi) =  \begin{pmatrix}
   & & & \bullet \\
   & B' & & \vdots\\
   & & & \bullet \\
    0&\cdots&0 & T_n
\end{pmatrix}.
\end{equation} 
Notice that entries of $B'$ and the ``$\bullet$" entries are in $k[T_1, \ldots, T_n]$ and $B'$ is the Jacobian dual of $\varphi'$ with respect to $x_1,\ldots,x_{d-1}$, i.e. $[T_1 \ldots T_n] \cdot \varphi' = [x_1 \ldots x_{d-1}]\cdot B'$. 
In the particular cases when $n=d+1$ or $d=3$, it was shown in \cite[5.1]{Nguyen14}, \cite[4.6]{Nguyen17} and \cite[3.6]{DRS18} that the non-linear equations defining the Rees algebra $\R(I)$ are precisely the generators of $I_{d-1}(B')$, under the additional assumption that $k$ is algebraically closed. We will prove in \Cref{defideal} that this actually holds for every ideal $I$ which satisfies the assumptions of \Cref{idealsetting}, with no restrictions on the field $k$. As a preliminary step, we identify the non-linear type locus of $I$ and prove that it has a unique minimal element. This recovers and generalizes what was known in the cases $d=3$ (see \cite[4.4 and 4.5]{Nguyen17} and \cite[2.5]{DRS18}) and $n=d+1$ (\cite[3.4 and 3.5]{Nguyen14}).

\begin{prop}\label{LTlocus}
    With $I$ and $\varphi$ as above, $I_\p$ is of linear type for all $\p \in \spec(R) \setminus V(I_{n-d+1}(\varphi))$. 
\end{prop}

\begin{proof}
Notice that $V(I_{n-d+1}(\varphi))= V(\fitt_{d-1}(\varphi)) = \{\p \in \spec(R) \,|\, \mu(I_\p) \geq d\}$ by \Cref{fitt}. Thus it follows that $I_\p$ satisfies $G_\infty$ and is hence of linear type for all $\p \notin V(\fitt_{d-1}(I))$ by \cite[9.1]{HSV81}.
\end{proof}

\begin{prop}\label{minFitt}
 With the assumptions of \Cref{idealsetting}, we have that $\p=(x_1,\ldots,x_{d-1})$ is the unique minimal prime of $I_{n-d+1}(\varphi)$.  
\end{prop}

\begin{proof}
   We modify the proof of \cite[4.5]{Nguyen17} according to the assumptions of \Cref{idealsetting}. First, notice that since $I$ satisfies $G_{d-1}$ but not $G_d$, we have $\hgt I_{n-d+2}(\varphi) = \hgt I_{n-d+1}(\varphi) = d-1$ by \Cref{hgtlfittlemma}. Now, let $\hat{\varphi}$ denote the $n \times (n-2)$ submatrix of $\varphi$ consisting of its first $n-2$ columns. As noted in \Cref{idealpres}, the entries of $\hat{\varphi}$ are linear forms in $k[x_1, \ldots,x_{d-1}]$ and also we have $I_{n-d+2}(\varphi) \subseteq I_{n-d+1}(\hat{\varphi}) \subseteq (x_1, \ldots, x_{d-1})$. Since $\hgt I_{n-d+2}(\varphi) = d-1 = \hgt (x_1, \ldots, x_{d-1})$, it follows that $\hgt I_{n-d+1}(\hat{\varphi})= d-1$. Thus $I_{n-d+1}(\hat{\varphi})$ is a $(x_1, \ldots, x_{d-1})$-primary ideal in $k[x_1,\ldots,x_{d-1}]$, and so $(x_1, \ldots, x_{d-1})$ is its only minimal prime. On the other hand, since $ I_{n-d+1}(\hat{\varphi}) \subseteq I_{n-d+1}(\varphi) \subseteq (x_1, \ldots, x_{d-1})$ and $\hgt I_{n-d+1}(\varphi) =d-1$, it follows that  $(x_1, \ldots, x_{d-1})$ is the unique minimal prime of $I_{n-d+1}(\varphi)$ as well.
\end{proof}

Our next task is to calculate the height of the ideal $I_{d-1}(B')$. First however, we describe the defining ideal $\J$ of the Rees algebra $\R(I)$ as a prime saturation. With $\varphi$ as in \Cref{idealpres}, recall from \Cref{prelimsection} that $\L= ([T_1,\ldots,T_n]\cdot \varphi)$ is the defining ideal of the symmetric algebra $\S(I)$.

\begin{prop}\label{Jasat}
 With the assumptions of \Cref{idealsetting}, $\J$ is a prime ideal of height $n-1$. Moreover, we have that $\J=\L:(x_1,\ldots,x_{d-1})^\infty$. 
\end{prop}

\begin{proof}
 The initial statement is clear as $\R(I)\cong R[T_1, \ldots, T_n]/\J$ and $\R(I)$ is a domain of dimension $d+1$. For the latter statement, let $y \in \L:(x_1,\ldots,x_{d-1})^\infty$, and so there exists an integer $s$ such that $y (x_1,\ldots,x_{d-1})^s \in \L \subseteq \J$. Since $(x_1,\ldots,x_{d-1})^s \nsubset \J$ and $\J$ is a prime ideal, it then follows that $y \in \J$, hence $\L:(x_1,\ldots,x_{d-1})^\infty \subseteq \J$.  To prove the reverse containment, consider the quotient $\A = \J/\L$. It suffices to show that $\A$ is annihilated by some power of $(x_1,\ldots,x_{d-1})$. From \Cref{LTlocus} and \Cref{minFitt}, it follows that $\A$ is supported only at $\p= (x_1,\ldots,x_{d-1})$ and $\m=(x_1,\ldots,x_{d})$. Thus $\p = \sqrt{\ann \A}$, and so indeed a power of $\p$ annihilates $\A$. 
\end{proof}

\begin{prop}\label{B'hgt}
With $B'$ the submatrix of $B(\varphi)$ defined in (\ref{JD}), $I_{d-1}(B')$ is a Cohen-Macaulay prime ideal of height $n-d$.     
\end{prop}

\begin{proof}
Recall that $B'$ is a $(d-1) \times (n-2)$ matrix with linear entries in $k[T_1,\ldots,T_n]$. Moreover, with the form of $B(\varphi)$ given in (\ref{JD}), we also have that $[\ell_1\ldots\ell_{n-2}] = [x_1\ldots x_{d-1}]\cdot B'$. By \Cref{Jasat} we have that $\L:(x_1,\ldots,x_{d-1})^\infty$ is a prime ideal of height $n-1$. From \cite[2.2]{BM16} it then follows that $(\ell_1,\ldots,\ell_{n-2}):(x_1,\ldots,x_{d-1})^\infty$ is a prime ideal of height $n-2$ (we note that the assumption that $(\ell_1,\ldots,\ell_{n-1})$ be contained in $(x_1,\ldots,x_{d-1})$ in the statement of \cite[2.2]{BM16} is superfluous, as can be seen in its proof). Hence, from \cite[2.4]{BM16} it follows that $I_{d-1}(B')$ is a prime ideal with $\hgt I_{d-1}(B') = (n-2)-(d-1)+1 = n-d$. Since this is the maximal possible height \cite[Thm.~1]{EN62}, the Cohen-Macaulayness of $I_{d-1}(B')$ then follows from \cite[A2.13]{Eisenbud}.
\end{proof}

We are now ready to prove that the generators of $I_{d-1}(B')$ are precisely the non-linear equations of $\R(I)$. 
To this end, we consider the auxiliary matrix
\begin{equation}\label{idealB''matrix}
 B'' =  \begin{pmatrix}
   & & & \bullet \\
   & B' & & \vdots\\
   & & & \bullet \\
    0&\cdots&0 & 1
\end{pmatrix}   
\end{equation}
which is obtained from the expression of $B(\varphi)$ given in (\ref{JD}) by dividing the last row by $T_n$. Notice that $[\ell_1\ldots \ell_{n-1}] = [x_1\ldots x_{d-1} \,\, x_dT_n] \cdot B''$ and also that $I_{d-1}(B') = I_d(B'')$. 

\begin{thm}\label{cmhgt}
 With the assumptions of \Cref{idealsetting}, $\L+ I_{d-1}(B')$ is a Cohen-Macaulay ideal of height $n-1$. 
\end{thm}

\begin{proof}
We proceed as in the first half of the proof of \cite[4.6]{Nguyen17}. Consider the ideal $\K = (x_1,\ldots,x_{d-1},x_d T_n)$ and notice that $\K \nsubseteq \J$. Hence $\L:\K \subseteq \J$ as $\L\subseteq \J$ and $\J$ is prime. Moreover, as $[\ell_1\ldots \ell_{n-1}] = [x_1\ldots x_{d-1} \,\, x_dT_n] \cdot B''$, from Cramer's rule it follows that $I_d(B'') \subseteq \L:\K$.
We show that $\L : \K$ is an $(n-1)$-residual intersection of $\K$, i.e. $\hgt \L:\K \geq n-1$. Since $x_1,\ldots,x_{d-1},x_d T_n$ is a regular sequence, from \Cref{thmResInt} one then obtains that $\L:\K$ is a Cohen-Macaulay ideal of height exactly $n-1$ and moreover, $\L:\K = \L + I_d(B'') = \L+I_{d-1}(B')$. 

Let $\q$ be a minimal prime of $\L:\K \subset R[T_1,\ldots,T_n]$ and let us consider two cases. If $(x_1,\ldots,x_{d-1})\nsubseteq \q$, then $\L_\q \subset (\L:\K)_\q \subseteq \J_\q = \L_\q$, where the last equality follows from \Cref{Jasat} and noting that $(x_1,\ldots,x_{d-1})_\q$ is the unit ideal. Thus $(\L:\K)_\q = \J_\q$ and since this is not the unit ideal, it follows that $\J \subseteq \q$, hence $\hgt \q \geq \hgt \J =n-1$. If instead $(x_1,\ldots,x_{d-1}) \subseteq \q$, as $I_d(B'') \subseteq \L:\K$, we then have that $(x_1,\ldots,x_{d-1}) + I_d(B'') \subseteq \q$. By \Cref{B'hgt}, $I_d(B'') = I_{d-1}(B')$ is a prime ideal of height $n-d$ in $k[T_1,\ldots,T_n]$. Therefore, $\hgt ((x_1,\ldots,x_{d-1}) + I_d(B'')) \geq (d-1) +(n-d) = n-1$, hence $\hgt \q \geq n-1$ in this case as well. Thus every minimal prime of $\L:\K$ has height at least $n-1$ and so $\hgt \L:\K \geq n-1$, as claimed. 
\end{proof}

\begin{thm}\label{defideal}
With the assumptions of \Cref{idealsetting}, the defining ideal of $\R(I)$ is 
$$\J = \L: (x_1,\ldots,x_{d-1}) =\L+I_{d-1}(B').$$
In particular, $\R(I)$ is Cohen-Macaulay.
\end{thm}

\begin{proof}
Notice that $\L+I_{d-1}(B') \subseteq \L:(x_1,\ldots,x_{d-1}) \subseteq \J$, where the first containment follows from Cramer's rule and the second from \Cref{Jasat}. Hence we need only prove that $\J = \L+I_{d-1}(B')$ and we note that it is enough to show that $\J_\p = (\L+I_{d-1}(B'))_\p$ for every associated prime of $\L+I_{d-1}(B')$. In fact, it suffices to show that any such associated prime does not contain $(x_1,\ldots,x_{d-1})$, as the thesis then follows from \Cref{LTlocus} and \Cref{minFitt}.

Let $\p$ be an associated prime of $\L+I_{d-1}(B')$ and note that $\p$ is actually a minimal prime of $\L+I_{d-1}(B')$ with $\hgt \p = n-1$ by \Cref{cmhgt}. It suffices to show that $(x_1,\ldots,x_{d-1}) + \p$ has height at least $n$, as then $\p$ does not contain $(x_1,\ldots,x_{d-1})$. To this end, notice that $(x_1,\ldots,x_{d-1},x_d T_n) = (x_1,\ldots,x_{d-1}) + \L$, hence $(x_1,\ldots,x_{d-1},x_d T_n) + I_{d}(B'') \subseteq (x_1,\ldots,x_{d-1}) + \p$ and so it is enough to show that $\hgt ((x_1,\ldots,x_{d-1},x_d T_n) + I_{d}(B'')) \geq n$. Recall that $I_{d}(B'')= I_{d-1}(B')$ is a prime ideal of $k[T_1,\ldots,T_n]$ with $\hgt I_d(B'')=n-d$ by \Cref{B'hgt}, hence it follows that $\hgt ((x_1,\ldots,x_d ) + I_{d}(B'')) \geq n$. Similarly, we have that $\hgt (I_{d}(B''),T_n) \geq n-d+1$ as $ I_{d}(B'')$ is prime and $T_n \notin I_{d}(B'')$, by degree considerations. As $(I_{d}(B''),T_n)$ is an ideal of $k[T_1,\ldots,T_n]$, it follows that $\hgt ((x_1,\ldots,x_{d-1}, T_n) + I_{d}(B'')) \geq n$ as well. 
\end{proof}

From \Cref{defideal} it follows that the defining equations of $\R(I)$ consist of the generators of $\L$ and non-linear equations belonging to $k[T_1,\ldots,T_n]$. With this, $I$ is said to be an ideal of \textit{fiber type}, as the latter equations generate the ideal defining the special fiber ring. Indeed, we have the following corollary.

\begin{cor}\label{F(I)defideal}
With the assumptions of \Cref{idealsetting}, the special fiber ring is $\F(I) \cong k[T_1,\ldots,T_n]/I_{d-1}(B')$. In particular, $\F(I)$ is a Cohen-Macaulay domain of dimension $d$, i.e. $I$ has maximal analytic spread $\ell(I)=d$.
\end{cor}

\begin{proof} 
Recall that the entries of $B'$ belong to $k[T_1, \ldots, T_n]$. From \Cref{defideal}, we have that $\J=\L+I_{d-1}(B')$ and so $\F(I) \cong k[T_1,\ldots,T_n]/I_{d-1}(B')$, as $\L\subseteq \m$. Recall from \Cref{B'hgt} that $I_{d-1}(B')$ is a Cohen-Macaulay prime ideal of $k[T_1,\ldots,T_n]$ with $\hgt I_{d-1}(B') = n-d$. Hence it follows that $\F(I)$ is a Cohen-Macaulay domain with $\dim \F(I) = d$.
\end{proof}

We remark that, in the case when $d=3$, the content of \Cref{F(I)defideal} was already proved in \cite[2.4~and~3.6]{DRS18} (or follows from \cite[4.3~and~4.6]{Nguyen17} with the same proof as above). Likewise, in the case that $n=d+1$, the result follows from \cite[5.2~and~5.3]{Nguyen14}.

We end this section by discussing possible directions to expand upon the results presented here, namely situations where assumptions (ii) or (iii) in \Cref{idealsetting} are not satisfied. In line with other examples provided in \cite{Nguyen17} and \cite{DRS18}, the examples presented here illustrate that similar behaviors of the ideal $I$ and its presentation matrix $\varphi$ should not be expected if these assumptions are weakened.

\subsection{Questions and possible generalizations} \label{Gssubsection}

A key ingredient in the proof of \Cref{defideal} is the fact that $\fitt_{d-1}(I)=I_{n-d+1}(\varphi)$ has a unique minimal prime $\p=(x_1 \ldots, x_{d-1})$ due to \Cref{minFitt}, whose proof relies on assumptions (ii) and (iii) in \Cref{idealsetting}. It is then natural to investigate whether a similar behavior persists if either of these assumptions is weakened.

One could ask whether the defining ideal of $\R(I)$ could still be determined by relaxing the rank condition in assumption (ii) in \Cref{idealsetting} while maintaining the other conditions. Note that this is only possible if $\mu(I) >d+1$, by \Cref{d+epres} and \cite[3.3]{Nguyen14}. A first observation is that if one completely removes condition (ii) from \Cref{idealsetting}, we are no longer guaranteed the result of \Cref{minFitt}. That is, we are no longer guaranteed that $I_{n-d+1}(\varphi)$ has a unique minimal prime ideal. When $d=3$, this follows from \cite[2.5]{DRS18}, but we provide an explicit example of such behavior, for the sake of completeness.

\begin{ex} \label{noMinPrimeExample}
Consider the following $5\times 4$ matrix with linear entries in $\mathbb{Q}[x_1,x_2,x_3]$:

$$\varphi=\left(\!\begin{array}{cccc}
x_{1}-x_{2}&x_{2}+x_{3}&0&x_{1}-x_{2}\\
0&0&x_{2}&0\\
x_{2}&x_{3}&x_{2}&0\\
x_{3}&0&x_{2}+x_{3}&x_{2}\\
x_{2}+x_{3}&0&x_{3}&x_{1}
\end{array}\!\right)$$
\medskip

Computations in \textit{Macaulay2} \cite{Macaulay2} show that the ideal $I=I_4(\varphi)$ is perfect of height two and satisfies $G_2$ but not $G_3$. Additionally, $I_{1}(\varphi) = (x_{1},x_{2},x_{3})$. However, $I_{3}(\varphi)$ has two minimal prime ideals, $(x_1,x_2)$ and $(x_2,x_3)$. Hence, \Cref{minFitt} cannot be satisfied, so assumption (ii) cannot hold. Indeed, a computation in \textit{Macaulay2} \cite{Macaulay2} shows that the rank of $\varphi$ modulo these primes is 2. 
\end{ex}

When $d=3$, the scenario that $\varphi$ has higher rank has been explored in \cite{DRS18}, via the notion of the \emph{chaos invariant} $u (\varphi)$ of $\varphi$ \cite[2.1]{DRS18}. This integer $u (\varphi)$ is equal to 1 when assumption (ii) holds. Moreover, in \cite[2.5]{DRS18} it was shown that certain Fitting ideals, in a range controlled by the chaos invariant, share a common unique minimal prime.  It is interesting to ask whether a similar behavior to \cite[2.5]{DRS18} occurs in more general settings, perhaps by generalizing the notion of the chaos invariant to arbitrary dimension $d$, and whether this may lead to determining the defining ideal of $\R(I)$ (see also \cite[3.3]{DRS18}).

It is also worth noting that assumption (ii) in \Cref{idealsetting} is sufficient, but not necessary, for $I_{n-d+1}(\varphi)$ to have a unique minimal prime ideal; see \Cref{Lanexample} below. It is then reasonable to ask whether we can recover some of the results of the present work if one replaces condition (ii) with the more general requirement that $I_{n-d+1}(\varphi)$ has a unique minimal prime. The following example suggests that, when doing so, the defining ideal of $\R(I)$ might be very different from the description in \Cref{defideal}, even in the three variable setting.

\begin{ex}[{\cite[4.7]{Nguyen17}}] \label{Lanexample}
Consider the following $5\times 4$ matrix with linear entries in $\mathbb{Q}[x_1,x_2,x_3]$:
$$\varphi=\left(\!\begin{array}{cccccc}
x_1&x_2&x_1&x_3\\
x_2&x_1&x_1-x_3&x_1\\
0&x_1&x_1-x_2&x_1\\
0&x_1&x_1&x_2\\
0&x_2&x_1&x_1
\end{array}\!\right)$$
\medskip

The ideal $I=I_4(\varphi)$ is perfect of height two and satisfies $G_2$ but not $G_3$. The ideal $I_{3}(\varphi)$ has a unique minimal prime ideal, $(x_1,x_2)$, and the rank of $\varphi$ modulo $(x_1,x_2)$ is 2. A computation through \textit{Macaulay2} \cite{Macaulay2} shows that the defining ideal $\J$ fails to be as in \Cref{defideal}, i.e. $\J\neq \L+I_{d-1}(B')$ as noted in \cite[4.7]{Nguyen17}. Moreover, we have that $\J \neq \L:(x_1,x_2)$, but rather that $\J=\L:(x_1,x_2)^2$. 
\end{ex}

In light of \Cref{Lanexample}, one may wonder whether the defining ideal has a similar description in terms of the rank $r$ of $\varphi$ modulo the unique minimal prime of $I_{n-d+1}(\varphi)$, provided such a unique minimal prime exists, even when $r \ge 2$. Notice that when $d=3$, this rank $r$ would coincide with the chaos invariant $u(\varphi)$ by \cite[2.2]{DRS18}. This has led to the following question.

\begin{quest}\label{biggerrankquest}
Let $I\subset k[x_1, \ldots, x_d]$ be a perfect ideal of height two with a linear presentation matrix $\varphi$ and suppose that $I$ satisfies $G_{d-1}$ but not $G_d$. If $I_{n-d+1}(\varphi)$ has a unique minimal prime $\p$ and if the presentation matrix $\varphi$ has rank $r \geq 2$ modulo $\p$, is the defining ideal of $\R(I)$ given by $\J= \L:\p^{r}$?
\end{quest}

Even if the question has an affirmative answer, the theory of residual intersections alone might not be sufficient to produce a generating set of $\J$ if $r \geq 2$. However, similar descriptions of the defining ideal were encountered in \cite{BM16,Weaver1,Weaver2}, where an iterative method was then employed to produce the equations of $\J$. It is interesting to ask if similar algorithmic techniques can be used to produce the equations of $\J$ if \Cref{biggerrankquest} has a positive answer.

We now discuss the case of weakening assumption (iii) in \Cref{idealsetting} and consider ideals satisfying $G_s$ but not $G_{s+1}$ for some $s < d-1$. The following examples, calculated using \textit{Macaulay2} \cite{Macaulay2}, show that these ideals exhibit different behavior in general, even when the remaining conditions in \Cref{idealsetting} are maintained.

\begin{ex}\label{Gs-ex1}
  Consider the following $9 \times 8$ matrix with linear entries in $\mathbb{Q}[x_1, \ldots, x_8]$: 
 $$\varphi= \left(\!\begin{array}{cccccccc}
0&0&0&0&0&0&0&x_{3}\\
0&x_{1}&0&x_{3}&0&0&0&0\\
0&0&0&0&x_{3}&0&x_{5}&x_{4}\\
0&0&x_{6}&0&x_{4}&0&0&0\\
0&x_{2}&0&x_{3}&0&0&0&x_{7}\\
x_{6}&0&x_{7}&x_{5}&0&x_{2}&x_{3}&0\\
x_{2}&x_{6}&x_{3}&x_{4}&0&0&x_{6}&0\\
x_{5}&0&0&x_{6}&x_{6}&0&x_{5}&x_{2}\\
x_{5}&0&x_{3}&x_{1}&x_{4}&x_{4}&x_{5}&x_{8}
\end{array}\!\right)$$

\medskip 

Notice that $I_{1}(\varphi) = \left(x_{1},\ldots,x_{8}\right)$ and that $\varphi$ has rank 1 modulo $(x_{1},\ldots,x_{7})$. Moreover, $I = I_{8}(\varphi)$ is a perfect ideal of height two satisfying $G_{6}$ but not $G_{7}$. However, $\fitt_{6}(I) = I_{3}(\varphi)$ has two minimal primes of height $6$, namely $\left(x_{2},x_{3},x_{4},x_{5},x_{6},x_{7}\right)$ and $\left(x_{1},x_{2},x_{3},x_{4},x_{5},x_{6}\right)$. 
\end{ex}

What is even more surprising is that, if $I$ satisfies $G_s$ but not $G_{s+1}$ for some $s<d-1$, there may be multiple minimal primes of $I_{n-s}(\varphi)$ of different heights, even in the case $n=d+1$. 

\begin{ex}\label{Gs-ex2}
Consider the following $7 \times 6$ matrix with linear entries in $\mathbb{Q}\left[x_{1},\ldots,x_{6}\right]$:
 $$\varphi = \left(\!\begin{array}{cccccc}
0&x_{4}&x_{3}-x_{4}&0&x_{6}&0\\
0&x_{1}-x_{2}&0&0&x_{4}+x_{5}&x_{3}\\
0&0&x_{1}-x_{2}&0&0&x_{3}\\
x_{3}-x_{4}&x_{2}&0&x_{1}&x_{6}&x_{4}\\
0&x_{5}&0&0&x_{4}+x_{5}&0\\
x_{2}+x_{3}&x_{4}&0&x_{1}&0&x_{4}\\
0&x_{5}&0&x_{4}&x_{4}&0
\end{array}\!\right)$$

\medskip

The ideal $I = I_{6}(\varphi)$ is perfect of height two and satisfies $G_{4}$ but not $G_{5}$. Moreover, $I_{1}(\varphi) = \left(x_{1},\ldots,x_{6}\right)$ and $\varphi$ has rank 1 modulo $(x_{1}, x_{2},x_{3},x_{4},x_{5})$. However, the ideal $\fitt_4(I)=I_{3}(\varphi)$ has two minimal primes of different heights, namely $(x_{1},x_{2},x_{3},x_{4})$ and $(x_{1}-x_{2},x_{3},x_{4},x_{5},x_{6})$.
\end{ex}

Despite these examples, computational experiments on \textit{Macaulay2} \cite{Macaulay2} show that, for a linearly presented height-two perfect ideal $I$ satisfying $G_s$ but not $G_{s+1}$, $\fitt_s(I)$ frequently does have a unique minimal prime of height $s$. In both \Cref{Gs-ex1} and \Cref{Gs-ex2}, the matrix $\varphi$ satisfies the rank condition, assumption (ii) of \Cref{idealsetting}, hence a natural question arises.

\begin{quest}
    Let $I\subset k[x_1, \ldots, x_d]$ be a perfect ideal of height two with a linear presentation matrix $\varphi$ and suppose that $I$ satisfies $G_s$ but not $G_{s+1}$ for some $s <d-1$. Can one identify a sufficient condition on the matrix $\varphi$, which guarantees that $\fitt_s(I) = I_{n-s}(\varphi)$ has a unique minimal prime of height $s$?
\end{quest}

A positive answer could potentially lead to a formulation of the defining ideal of $\R(I)$ for $I$ above, so long as one can describe the non-linear type locus of $I$, as in \Cref{LTlocus}.

\section{Linearly presented modules of projective dimension one}\label{modulesection}
 
In this section we aim to extend the main result of the previous section, \Cref{defideal}, to Rees algebras of modules with projective dimension one. For such a module $E$, we determine the defining ideal $\J$ of its Rees algebra $\R(E)$ with the assumption that $E$ satisfies $G_{d-1}$ but not $G_d$. More precisely, we consider the following setting.

\begin{set}\label{modulesetting}
    Let $R=k[x_1,\ldots,x_d]$ be a standard-graded polynomial ring over a field $k$, with $\m = (x_1,\ldots,x_d)$ and $d\geq 3$. Let $E$ be a finitely generated $R$-module of projective dimension one, rank $e>0$, and $\mu(E)=n\geq d+e$, satisfying the following assumptions:  

    \begin{itemize}
        \item[(i)] The module $E$ has a minimal presentation matrix $\varphi$ consisting of linear entries with $I_1(\varphi) = \m$.
        \item[(ii)] After possibly a change of coordinates, the matrix $\varphi$ has rank 1 modulo an ideal generated by $d-1$ variables.
        \item[(iii)] The module $E$ satisfies $G_{d-1}$ but does not satisfy $G_d$. 
    \end{itemize}
\end{set}

\begin{rem}\label{E is torsion-free}
For similar reasons to those in \Cref{datleast3}, it is assumed that $d\geq 3$ in \Cref{modulesetting}. Moreover, as $d\geq 3$, from assumption (iii) above it follows that $E$ satisfies $G_2$ and is hence free locally in codimension one, thus $E$ is torsion-free. Hence, by \Cref{existBourbaki} we are guaranteed that a generic Bourbaki ideal $I$ exists. Notice also that if $e=1$, then $E$ is isomorphic to a perfect ideal of height $2$ and \Cref{idealsetting} is recovered. 

Furthermore, it assumed that $\mu(E) \geq d+e$ to avoid contradicting the remaining assumptions of \Cref{modulesetting}, as similarly noted in \Cref{datleast3} in the case of ideals. Additionally, if $\mu(E) \leq d+e-2$ and $E$ satisfies $G_{d-1}$, then $E$ is $G_{\infty}$ and hence of linear type by \cite[Prop.~3~and~4]{Avramov81}. If  $\mu(E) = d+e-1$, then since $E$ satisfies $G_{d-1}$ but not $G_d$ we must have $\hgt I_{1}(\varphi) =d-1$ by \Cref{hgtlfittlemma}, hence $E$ is a module over a polynomial subring $R' \subset R$ in $d-1$ variables. Thus, the defining ideal of $\R(E)$ is known by \cite[4.11]{SUV03}.
\end{rem}

\begin{rem}\label{modulepres}
Just as in \Cref{idealpres}, assumption (ii) in \Cref{modulesetting} ensures that, after an appropriate change of coordinates, we are able to assume that $\varphi$ has rank one modulo the ideal $(x_1,\ldots,x_{d-1})$. Then, after potential row and column operations, the presentation matrix $\varphi$ is of the form 
\begin{equation}\label{modulephiform}
 \varphi =  \begin{pmatrix}
   & & & *\\
   & \varphi'& & \vdots\\
   & & & *\\
    *&\cdots& *& x_d
\end{pmatrix}  
\end{equation}
where the entries of $\varphi'$ and the ``$*$" entries are linear and contained in the subring $k[x_1,\ldots,x_{d-1}]$. 
\end{rem}

Similarly as for ideals in the previous section, the form of the presentation matrix in (\ref{modulephiform}) obtained from assumption (ii) in \Cref{modulesetting} will be essential in order to describe the defining ideal of the Rees algebra $\R(E)$. We note however, that this condition is implied from the remaining assumptions when $\mu(E) = d+e$. A similar observation was made for perfect ideals of grade two in \cite[3.3]{Nguyen14}, however its proof relies on the additional assumption that $k$ is algebraically closed. We present a proof here for modules of projective dimension 1, without this assumption, in our current setting. As this argument applies to perfect ideals of grade two in the case that $e=1$, it follows that the assumption that $k$ is algebraically closed is superfluous in the settings of \cite{Nguyen14,Nguyen17}.

\begin{prop}\label{d+epres}
Assumption (ii) in \Cref{modulesetting} is automatically satisfied under the remaining assumptions if $\mu(E) = d+e$.
\end{prop}

\begin{proof}
Since $E$ satisfies the condition $G_{d-1}$ but not $G_d$ and $\mu(E) =d+e$, from \Cref{hgtlfittlemma} it follows that $\hgt I_3(\varphi) = \hgt I_2(\varphi) = d-1$. Moreover, since $E$ satisfies $G_{d-1}$ and not $G_d$, we have that $\mu(E_\q) \leq \dim R_\q+e-1$ for all primes $\q$ with $1 \leq \dim R_\q \leq d-2$, and also that there exists a prime $\p$ with $\hgt \p = d-1$ such that $\mu(E_\p)\geq \dim R_\p+e = d+e-1$. Let 
\begin{equation}\label{resofE}
 0\longrightarrow  R^d \overset{\varphi}{\longrightarrow} R^{d+e} \longrightarrow E \longrightarrow 0
\end{equation}
denote the minimal presentation of $E$ and let $\varphi_\p$ denote the image of $\varphi$ in (\ref{resofE}) after localizing at $\p$. Since $I_1(\varphi)= \m$, there is a nonzero entry of $\varphi$ not contained in $\p$. Hence $\varphi_\p$ has a unit entry and so it follows from \Cref{fitt} that $\mu(E_\p) \leq d+e-1$, hence $\mu(E_\p)=d+e-1$. Additionally from \Cref{fitt}, we have that $\{\q\in \spec(R)\,|\,\mu(E_\q)=d+e-1\} = V(I_2(\varphi))\setminus V(I_1(\varphi))$. As $I_1(\varphi) = \m$ and $\hgt I_2(\varphi) = d-1$, this set is exactly the set of minimal primes of $I_2(\varphi)$, hence $\p$ is a minimal prime of $I_2(\varphi)$. 

As previously noted, there is an entry of $\varphi$ not contained in $\p$. Thus, after potential row and column operations, we may assume that 
\begin{equation}\label{phiwhend+e}
 \varphi =  \begin{pmatrix}
   & & & *\\
   & \varphi'& & \vdots\\
   & & & *\\
    *&\cdots& *& \alpha
\end{pmatrix}  
\end{equation}
where $\alpha \notin \p$ and the entries of $\varphi'$ and the ``$*$" entries are linear forms in $R$. Now since $\mu(E_\p) = d+e-1$, it follows that $\varphi'_\p$ has no unit entry by \Cref{fitt}, and so $I_1(\varphi') \subseteq \p$. As $I_3(\varphi) \subseteq I_1(\varphi') \subseteq \p$ and $\hgt I_3(\varphi) = d-1$, it follows that $I_1(\varphi')$ is an ideal of height $d-1$ generated by linear forms. Thus, after a possible change of coordinates, we may assume that $\p = I_1(\varphi') = (x_1,\ldots,x_{d-1})$ and $\alpha =x_d$. Letting $\overline{\varphi}$ denote the image of $\varphi$ modulo $\p$, we have $$d+e-1 = \mu(E_\p) = {\rm dim}_{k(\p)} (E/\p E)_\p = d+e - \rank \overline{\varphi},$$ where $k(\p) = R_\p/\p R_\p$, the residue field of $R_\p$. Thus it follows that $\rank \overline{\varphi} =1$, as claimed.
\end{proof}

With the form of the presentation matrix $\varphi$ given in \Cref{modulepres}, the Jacobian dual of $\varphi$ has the following form 
\begin{equation}\label{moduleJD}
 B(\varphi) =  \begin{pmatrix}
   & & & \bullet \\
   & B' & & \vdots\\
   & & & \bullet \\
    0&\cdots&0 & T_n
\end{pmatrix}
\end{equation} 
where the entries of $B'$ and the ``$\bullet$" entries are linear forms in $k[T_1,\ldots,T_n]$. Once again, notice that $B' = B(\varphi')$ is the Jacobian dual of $\varphi'$ with respect to $x_1,\ldots,x_{d-1}$, i.e. $[T_1 \ldots T_n] \cdot \varphi' = [x_1 \ldots x_{d-1}]\cdot B'$.

As before, let $\J$ denote the defining ideal of $\R(E)$. Recall from \Cref{prelimsection} that $\L= ([T_1,\ldots,T_n]\cdot \varphi)$ is the defining ideal of the symmetric algebra $\S(E)$.

\begin{cor}\label{modulecontainment}
We have the containment of ideals $\L + I_{d-1}(B') \subseteq \J$.    
\end{cor}

\begin{proof}
By Cramer's rule, we have that $I_d(B(\varphi)) \subseteq\J$. Moreover, from (\ref{moduleJD}) we also have that $I_d(B(\varphi)) = (T_n)I_{d-1}(B')$. The claim then follows as $\J$ is a prime ideal and $T_n\notin \J$.
\end{proof}

In the remainder of this section, we aim to show that the containment in \Cref{modulecontainment} is actually an equality, for which we employ a generic Bourbaki ideal of $E$.

\subsection{The defining ideal of $\R(E)$}

Our next step is to show that, for a module $E$ as in \Cref{modulesetting}, the structure of the presentation matrix $\varphi$ allows one to select a generic Bourbaki ideal $I$ which satisfies the assumptions of \Cref{defideal}. We emphasize that a generic Bourbaki ideal exists due to \Cref{existBourbaki}, as $E$ is torsion-free and satisfies $G_2$. For the duration of this section, we adopt the following notation. 

\begin{notat}
  Retain the assumptions and notations of \Cref{modulesetting} and assume that $\varphi$ is expressed as in \Cref{modulepres}. Let $I$ denote a generic Bourbaki ideal of $E$ obtained as in \Cref{gbipresentation}, with presentation matrix $\psi$. Denote by $\L_I$ and $\J_I$ the defining ideals of $\S(I)$ and of $\R(I)$, respectively.
\end{notat}

\begin{prop}\label{choosegbi}
    A generic Bourbaki ideal $I$ may be chosen to satisfy the assumptions of \Cref{idealsetting}.
\end{prop}

\begin{proof}
Let $I$ be a generic Bourbaki ideal constructed as in \Cref{gbipresentation} and notice that $I$ is linearly presented since $E$ is.
Moreover, by \Cref{existBourbaki} it follows that $I$ satisfies $G_{d-1}$, but not $G_d$. Finally, the Bourbaki sequence
    $$0\longrightarrow F'' \longrightarrow E'' \longrightarrow I\longrightarrow 0$$
    shows that the projective dimension of $I$ is at most $1$ and hence exactly $1$, since $E''$ is not free. As $E$ is orientable, we may also take $I$ to have grade at least $2$ by \Cref{existBourbaki}. Thus $I$ is a perfect ideal of grade 2, generated by $\mu(I) = n-e+1 \geq d+1$ homogeneous elements. Furthermore, if $\psi$ denotes the presentation matrix of $I$, we may assume that $I_1(\psi) = \m R''$. Lastly, \Cref{modulepres} and \Cref{gbipresentation} show that modulo $(x_1,\ldots,x_{d-1})$, $\psi$ has rank $1$.    
\end{proof}

As $I$ satisfies the assumptions of \Cref{idealsetting} of the previous section, its Rees algebra $\R(I)$ is well-understood. In particular, $\R(I)$ is Cohen-Macaulay and its defining ideal is known by \Cref{defideal}. 

\begin{prop}\label{deformation}
    Let $E$ be as in \Cref{modulesetting} and let $I$ be a generic Bourbaki ideal of $E$ as in \Cref{choosegbi}. In accordance with \Cref{NotationBourbaki}, let $a_1, \ldots, a_n$ be a generating set of $E$ and, for $1\leq j \leq e-1$, write $y_j = \sum_{i=1}^n Z_{ij} a_i \in E''$ and $Y_j = \sum_{i=1}^n Z_{ij} T_i \in R''[T_1,\ldots,T_n]$. The following statements hold.
    
    \begin{itemize}
        \item[{\rm(a)}] The Rees algebra $\R(E)$ is Cohen-Macaulay.  
        \item[{\rm(b)}] There is an isomorphism $\R(E'')/(F'') \cong \R(I)$. Moreover,  $y_1,\ldots,y_{e-1}$ form a regular sequence on $\R(E'')$.    
        \item[{\rm(c)}] The defining ideal $\J_I$ of $\R(I)$ is $$\J_I = \J R''[T_1,\ldots,T_n] + (Y_1,\ldots,Y_{e-1})$$
    where $\J$ is the defining ideal of $\R(E)$. Additionally, $Y_1,\ldots,Y_{e-1}$ form a regular sequence modulo $\J R''[T_1,\ldots,T_n]$. 
    \end{itemize}
 \end{prop}

\begin{proof}
    By \Cref{choosegbi}, the generic Bourbaki ideal $I$ satisfies the assumptions of \Cref{defideal}, hence $\R(I)$ is Cohen-Macaulay. Assertions (a) and (b) then follow from \Cref{MainBourbaki}. Additionally, notice that the natural map $\pi \colon R''[T_1,\ldots,T_n] \twoheadrightarrow \R(E'')$ maps $Y_1,\ldots,Y_{e-1}$ to $y_1,\ldots,y_{e-1}$, so (c) follows from (b).  
\end{proof}

With \Cref{deformation}(b), one says that $\R(E'')$ is a \textit{deformation} of $\R(I)$. We now show that the shape of the defining ideal is actually preserved under this deformation. To this end, the specific form of $\J_I$ given in \Cref{defideal} will be crucial.

\begin{rem}
Notice that, with the presentation matrix $\varphi$ of $E$ as in \Cref{modulepres}, the presentation matrix $\psi$ of $I$ can be constructed to have a similar shape.  Indeed, from \Cref{gbipresentation} and \Cref{modulepres}, it follows that $\psi$ has the form
\begin{equation}\label{presmatrixpsi}
 \psi =  \begin{pmatrix}
   & & & *\\
   & \psi'& & \vdots\\
   & & & *\\
    *&\cdots& *& x_d
\end{pmatrix}
\end{equation}
where the entries of $\psi'$ and the ``$*$" entries are linear and contained in the subring $k(Z)[x_1,\ldots,x_{d-1}]$ of $R''$. With this, the Jacobian dual of $\psi$ has the form
\begin{equation}\label{modulepsiJDeqn}
 B(\psi) =  \begin{pmatrix}
   & & & \bullet \\
   & B_I' & & \vdots\\
   & & & \bullet \\
    0&\cdots&0 & T_n
\end{pmatrix},
\end{equation}
where the entries of $B_I'$ and the ``$\bullet$" entries are linear forms in $k(Z)[T_1,\ldots,T_n]$. By \Cref{defideal}, the defining ideal of $\R(I)$ is $\J_I = \L_I +I_{d-1}(B_I')$.
\end{rem}

\begin{lemma}\label{keylemma}
With $Y_1,\ldots,Y_{e-1}$ as in \Cref{deformation}, we have 
$$\L+I_{d-1}(B') + (Y_1,\ldots,Y_{e-1}) = \L_I + I_{d-1}(B_I')$$
in $R''[T_1,\ldots,T_n]$.
\end{lemma}

\begin{proof}
As noted in the proof of \Cref{deformation}, we have that each $Y_j$ maps to $y_j$ under the natural map $\pi \colon R''[T_1,\ldots,T_n] \twoheadrightarrow \R(E'')$. Moreover, taking $y_1,\ldots,y_{e-1}$ to be part of a minimal generating set of $E''$, as in \Cref{gbipresentation}, modulo $Y_1,\ldots,Y_{e-1}$ we have
    $$[x_1\ldots x_d] \cdot B(\varphi)  \equiv  [T_1\ldots T_n] \cdot \begin{bmatrix} \hspace{2mm}0\hspace{2mm} \\
\psi\end{bmatrix} \mod\, (Y_1,\ldots,Y_{e-1}).$$
In particular, since $[T_1\ldots T_n] \cdot \psi = [x_1\ldots x_d] \cdot B(\psi)$, we have
$$[x_1\ldots x_d]\cdot B(\varphi) \equiv  [x_1\ldots x_d]\cdot B(\psi) \mod\, (Y_1,\ldots,Y_{e-1}),$$
from which it follows that 
\begin{equation}\label{Ldeforms}
    \L + (Y_1,\ldots,Y_{e-1}) = \L_I .
\end{equation}

Now that we have shown that the shape of $\L$ and $\L_I$ is preserved under the deformation of \Cref{deformation}, we now show that the same holds for $\L+I_{d-1}(B')$ and $\L_I+I_{d-1}(B_I')$. Consider the auxiliary matrices 
\[
B'' =  \begin{pmatrix}
   & & & \bullet \\
   & B' & & \vdots\\
   & & & \bullet \\
    0&\cdots&0 & 1
\end{pmatrix}\quad \quad  {\rm and} \quad \quad 
B_I'' =  \begin{pmatrix}
   & & & \bullet \\
   & B_I' & & \vdots\\
   & & & \bullet \\
    0&\cdots&0 & 1
\end{pmatrix}
\]
obtained by dividing the last row of $B(\varphi)$ and $B(\psi)$ by $T_n$, respectively, just as in the construction outlined prior to \Cref{cmhgt}. Notice that $\L = ([x_1\ldots x_{d-1} \,\, x_dT_n] \cdot B'')$ and $\L_I = ([x_1\ldots x_{d-1} \,\, x_dT_n] \cdot B_I'')$. As we had seen that $\L + (Y_1,\ldots,Y_{e-1}) = \L_I$ in (\ref{Ldeforms}) and $x_1,\ldots, x_{d-1}, x_dT_n$ is a regular sequence, it follows from \cite[4.4]{BM16} that $$\L+I_d(B'') + (Y_1,\ldots,Y_{e-1}) = \L_I + I_d(B_I'').$$ However, notice that $I_d(B'') = I_{d-1}(B')$ and $I_d(B_I'') = I_{d-1}(B_I')$, from which the claim follows.
\end{proof}

With \Cref{keylemma}, we are finally able to describe the defining ideal of $\R(E)$ and provide a module analogue of \Cref{defideal}.

\begin{thm}\label{defidealmodule}
With $E$ as in \Cref{modulesetting}, the defining ideal of $\R(E)$ is 
$$\J = \L: (x_1,\ldots,x_{d-1}) =\L+I_{d-1}(B').$$  
\end{thm}

\begin{proof}
The containments $\L+I_{d-1}(B')\subseteq \L: (x_1,\ldots,x_{d-1}) \subseteq \J$ are clear, hence it suffices to show that $\J = \L+I_{d-1}(B')$. Recall that $\J_I = \J R''[T_1,\ldots,T_n] + (Y_1,\ldots,Y_{e-1})$ by \Cref{deformation} and additionally, $\J_I = \L_I+I_{d-1}(B_I')$ by \Cref{choosegbi} and \Cref{defideal}. Thus by \Cref{keylemma}, we have
$$\J R''[T_1,\ldots,T_n] +(Y_1,\ldots,Y_{e-1}) = \L+I_{d-1}(B') + (Y_1,\ldots,Y_{e-1}).$$
Since $Y_1,\ldots,Y_{e-1}$ is a regular sequence modulo $\J R''[T_1,\ldots,T_n]$, we have that
\begin{align*}
  \J R''[T_1,\ldots,T_n] =& \big(\J R''[T_1,\ldots,T_n] +(Y_1,\ldots,Y_{e-1})\big) \cap \J R''[T_1,\ldots,T_n]\\[1ex]  
  =& \big(\L+I_{d-1}(B') + (Y_1,\ldots,Y_{e-1})\big) \cap \J R''[T_1,\ldots,T_n]\\[1ex]  
  =& \big(\L+I_{d-1}(B')\big) + (Y_1,\ldots,Y_{e-1}) \cap \J R''[T_1,\ldots,T_n]\\[1ex]
  =& \big(\L+I_{d-1}(B')\big) + (Y_1,\ldots,Y_{e-1})\J R''[T_1,\ldots,T_n].
\end{align*}
Nakayama's lemma then implies that $\J R''[T_1,\ldots,T_n] =\L+I_{d-1}(B')$ in $R''[T_1,\ldots,T_n]$, hence $\J=\L+I_{d-1}(B')$ in $R[T_1,\ldots,T_n]$, as claimed.
\end{proof}

\begin{cor}\label{F(E)cor}
With $E$ as in \Cref{modulesetting}, the special fiber ring is $\F(E) \cong k[T_1,\ldots,T_n]/I_{d-1}(B')$. Moreover, $\F(E)$ is a Cohen-Macaulay domain of dimension $d+e-1$, hence $E$ has maximal analytic spread $\ell(E)= d+e-1$.
\end{cor}

\begin{proof}
From \Cref{defidealmodule}, we have that $\J =\L+I_{d-1}(B')$. The first claim then follows as in \Cref{F(I)defideal}, as the entries of $B'$ belong to $k[T_1, \ldots, T_n]$ and $\L \subseteq \m$. Moreover, it follows that $\dim \F(E) = \ell(E) = d+e-1$ from \Cref{F(I)defideal} and \cite[3.10(a)]{SUV03}. Lastly, since $\R(I)$ and $\F(I)$ are simultaneously Cohen-Macaulay by \Cref{defideal} and \Cref{F(I)defideal}, it follows from \Cref{CMfiber} that $\F(E)$ is Cohen-Macaulay. 
 \end{proof}

 As a consequence of the proof of \Cref{defidealmodule} and \Cref{d+epres}, in the case that $\mu(E) = d+e$ we may remove assumption (ii) from \Cref{modulesetting} to obtain the following module analogue of \cite[5.2]{Nguyen14}. As noted in the discussion prior to \Cref{d+epres}, the assumption that $k$ be algebraically closed in the statement of \cite[5.2]{Nguyen14} may be removed. We note that the omission of this assumption is critical if one intends to employ a generic Bourbaki ideal to prove the result for modules. Indeed, even if the field $k$ is algebraically closed, its extension $k(Z)$ is not, where $Z$ is the set of indeterminates as in \Cref{NotationBourbaki}.

\begin{thm}\label{d+eVersion}
  Let $R = k[x_1, \ldots, x_d]$ be a polynomial ring over a field $k$ with $\m = (x_1, \ldots,x_d)$ and $d \geq 3$. Let $E$ be a finitely generated $R$-module of projective dimension one and rank $e>0$ admitting a minimal presentation matrix $\varphi$ consisting of linear entries and such that $I_1(\varphi) = \m$. Assume that $\mu(E) = d + e$, and that $E$ satisfies $G_{d-1}$ but does not satisfy $G_d$. The defining ideal of $\R(E)$ is
   $$\J= \L+\det(B').$$
  Moreover, both $\R(E)$ and the special fiber ring $\F(E)$ are Cohen-Macaulay.
\end{thm}

\begin{proof}
As noted in \Cref{d+epres}, the presentation matrix $\varphi$ may be taken to have the form given in (\ref{modulephiform}) in this setting. Moreover, as in the proof of \Cref{choosegbi}, a generic Bourbaki ideal $I$ may be taken to satisfy the assumptions of \cite[5.2]{Nguyen14}. As $\R(I)$ is Cohen-Macaulay by \cite[5.1]{Nguyen14}, \Cref{MainBourbaki} implies that the deformation condition in \Cref{deformation}(b) is satisfied. Hence, proceeding as in the proof of \Cref{defidealmodule}, one obtains that $\J= \L+\det(B')$. 

The Cohen-Macaulayness of $\R(E)$ follows from \Cref{MainBourbaki}, as $\R(I)$ is Cohen-Macaulay. Lastly, since $\mu(E) = d+e$, the submatrix $B'$ of the Jacobian dual $B(\varphi)$ is a square $(d-1)\times (d-1)$ matrix. Its determinant is a degree-$(d-1)$ homogeneous polynomial in $k[T_1,\ldots,T_n]$, and so $(x_1,\ldots,x_d)+\J = (x_1,\ldots,x_d)+ (\det B')$. Thus, $\F(E)$ is a hypersurface ring defined by the principal ideal $(\det(B'))$, and is hence Cohen-Macaulay.
\end{proof}

\subsection*{Acknowledgements}
The authors thank Prof.~Aron Simis for his valuable comments on a preliminary version of this paper which improved the exposition of \Cref{Gssubsection}, and for pointing out the Cohen-Macaulayness of the ideal in \Cref{B'hgt}. A. Costantini was partially supported by an AMS-Simons Travel Grant.

\end{document}